\documentclass[11pt]{article}

\usepackage{amsmath,amssymb,amsthm}
\usepackage{tikz,graphicx}
\usepackage{pstricks}
\usepackage{float}
\usepackage[utf8]{inputenc}
\usepackage[english]{babel}
\usepackage[T1]{fontenc}
\usepackage{lmodern}
\usepackage[totalwidth=160mm, totalheight=247mm]{geometry}

\renewcommand{\leq}{\leqslant}
\renewcommand{\geq}{\geqslant}

\newcommand{\bbZ}{\mathbb{Z}}

\newcommand{\bbR}{\mathbb{R}}

\newcommand{\cl}{\rm{\bf Cl}}

\newcommand{\Dc}{\mathcal{D}}
\newcommand{\Fc}{\mathcal{F}}

\newcommand{\Kc}{\mathcal{K}}
\newcommand{\Uc}{\mathcal{U}}

\newcommand{\mi}{\mathbf{i}}

\theoremstyle{plain}
\newtheorem{theo}{Theorem}[section]
\newtheorem{main-theorem}{Theorem}

\newtheorem{lemm}[theo]{Lemma}
\theoremstyle{definition}
\newtheorem{defi}[theo]{Definition}
\newtheorem{exam}[theo]{Example}
\newtheorem{conj}[theo]{Conjecture}
\newtheorem{rem}[theo]{Remark}

\newtheorem*{clai-nn}{Claim}

\begin{document}
\title{Every component of a fractal square is a Peano continuum\thanks{This study is partially supported by Chinese NSF projects 11171123,11431007, 11471124 and 11771153.}}
\author{Jun Luo, Hui Rao and Ying Xiong}
\maketitle
\begin{abstract}
This paper concerns the local connectedness of components of self-similar sets.
Let   $n\ge2$ be an integer  and let $\Dc\subset\{0,1,\ldots,n-1\}^2$.
The self-similar set $F$ satisfying $\displaystyle F=\bigcup_{d\in\Dc}(F+d)/n$
is called a {\em fractal square}.
 We prove that every component of $F$ is locally connected. The same result for three-dimensional analogues of $F$ does not hold.
\end{abstract}


\section{Introduction}

An iterated function system (IFS) on  $\bbR^n$ is a family $\Fc=\{f_1,\ldots,f_q\}$  of  contractions $f_j:\bbR^n\rightarrow\bbR^n$ with $q\ge2$. The unique  nonempty compact set $E\subset\bbR^n$ with $E=\bigcup_jf_{j}(E)$ is called the \emph{attractor} of $\Fc$, or a \emph{self-similar set} determined by $\Fc$.  If all the mappings $f_j$ are similitudes, then we call
$E$ a {\em strictly self-similar set} (\cite{Hutchinson}).

There are many works devoted to the topological properties of self-similar sets, mainly on (a) connectedness, see for instance \cite{Deng-Lau,Hata,LAT,Ngai-Tang,Taylor-2011}  or the survey \cite{AT} and the references therein; (b) homeomorphy to a disk, see for instance \cite{LL,LLT,LRT,LZ-2004}; or even (c) homeomorphy to a ball (\cite{KLT,CT-2016}) and (d) the fundamental group (\cite{ADTW,DT,JLL, LT}).
However, local connectedness of self-similar sets is touched occasionally (\cite{Tang,Luo-2007}).

In 1985, M. Hata \cite{Hata} showed  that if a
self-similar set is connected, then it  is also locally connected, and hence  is a locally connected continuum,  also called a \emph{Peano continuum}.
This paper concerns two questions:

\textbf{Q1.} \emph{When is a self-similar set locally connected ?}

\textbf{Q2.} \emph{When is every component of a self-similar set locally connected ?}

 For the first question, we have the following simple answer.

 \begin{main-theorem}\label{main-1}
Let $E\subset \bbR^n$ be a self-similar set. Then $E$ is locally connected if and only if it has finitely many connected components. In this case $E$ consists of finitely many Peano continua.
\end{main-theorem}

The second question is much more delicate  and it is the main concern
of the present paper. We begin with  examples.

\begin{exam}\label{comb} Let $E_1={\mathcal K}\times [0,1]$ where ${\mathcal K}$ is the $1/4$-Cantor set. Clearly $E_1$ can be generated by the IFS $\displaystyle\Fc_1=\left\{  \frac{x+d}{4};~\ d\in D \right\}$  where
$\displaystyle D=\left \{a+b \mi; a=0,3; b=0,1,2,3  \right\}.$ Here $\mi$ denotes the imaginary unit, identified with the vector $(0,1)\in\mathbb{R}^2$.
The mappings in the above IFS may be illustrated by Figure \ref{cantor-comb}(left and middle), where the arrows  in the squares indicate the rotations and reflections involved. It is seen that $E_1$ is not locally connected, but its components are line segments and hence are locally connected.
\begin{figure}[h]
\begin{center}\vskip -0.1cm
\begin{tabular}{ccc}
\includegraphics[width=4.5cm]{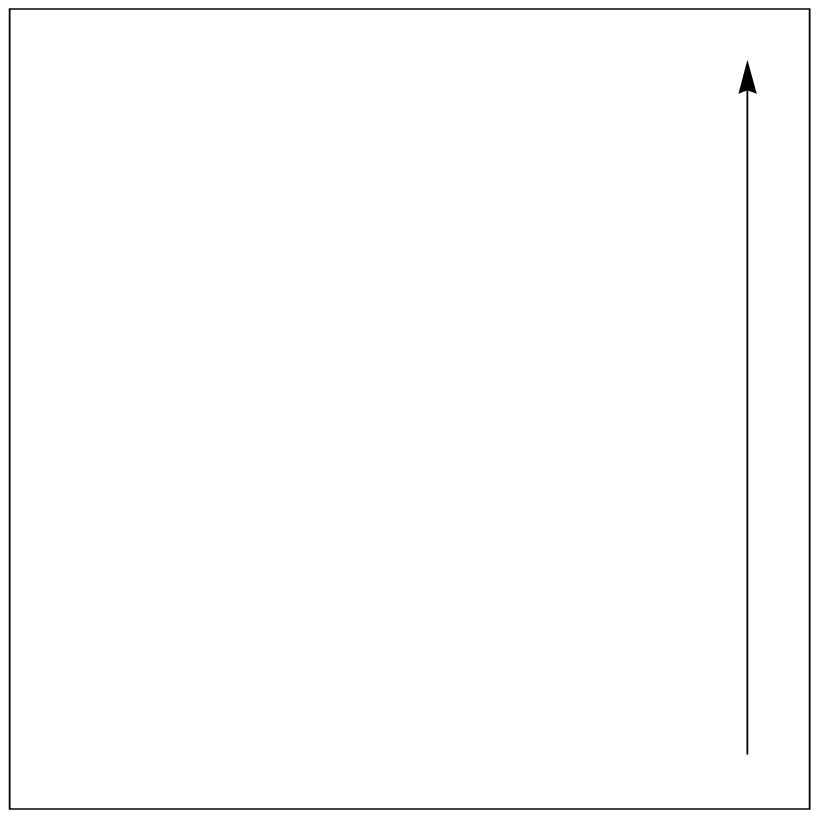}& \includegraphics[width=4.5cm]{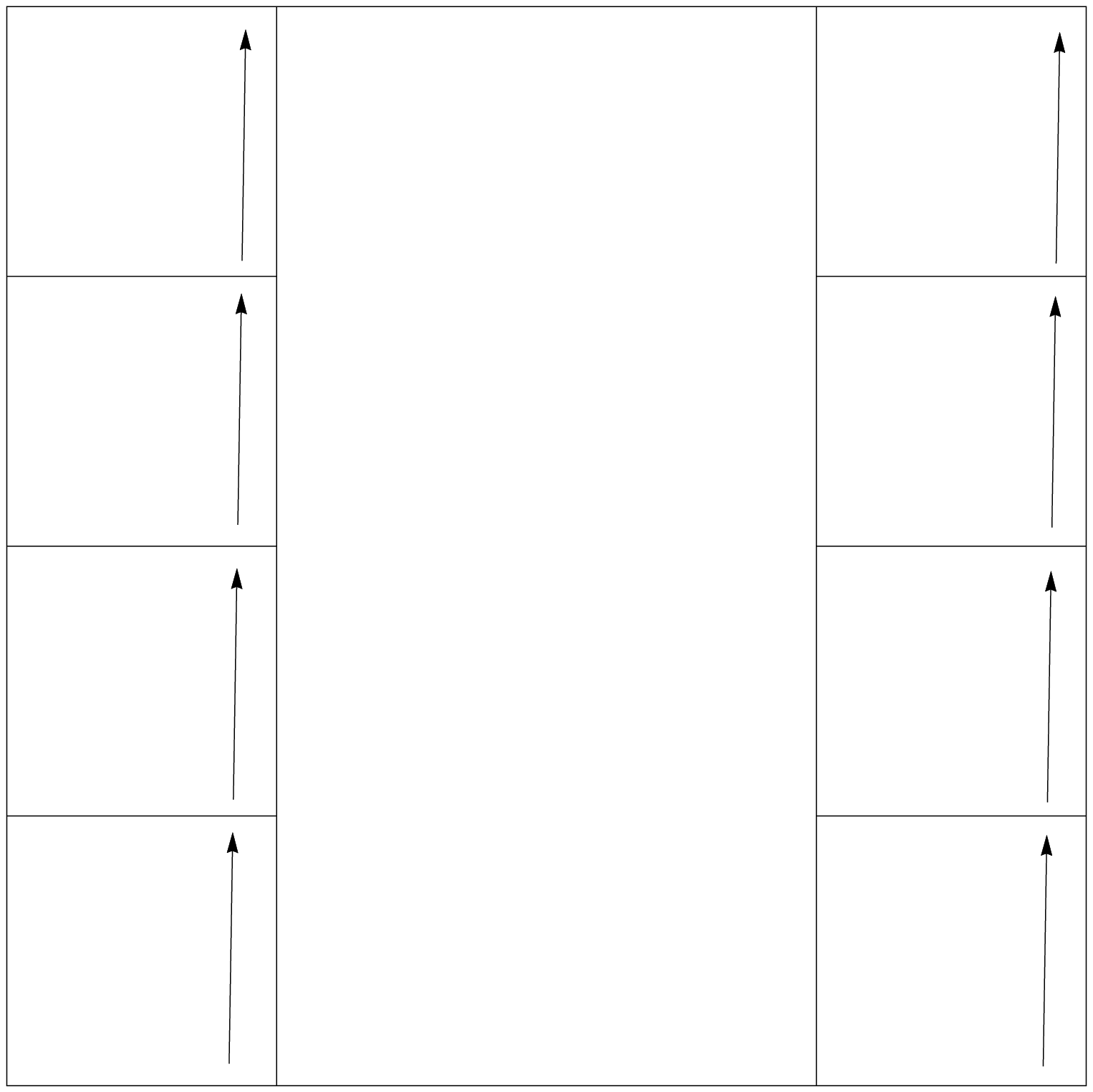}& \includegraphics[width=4.5cm]{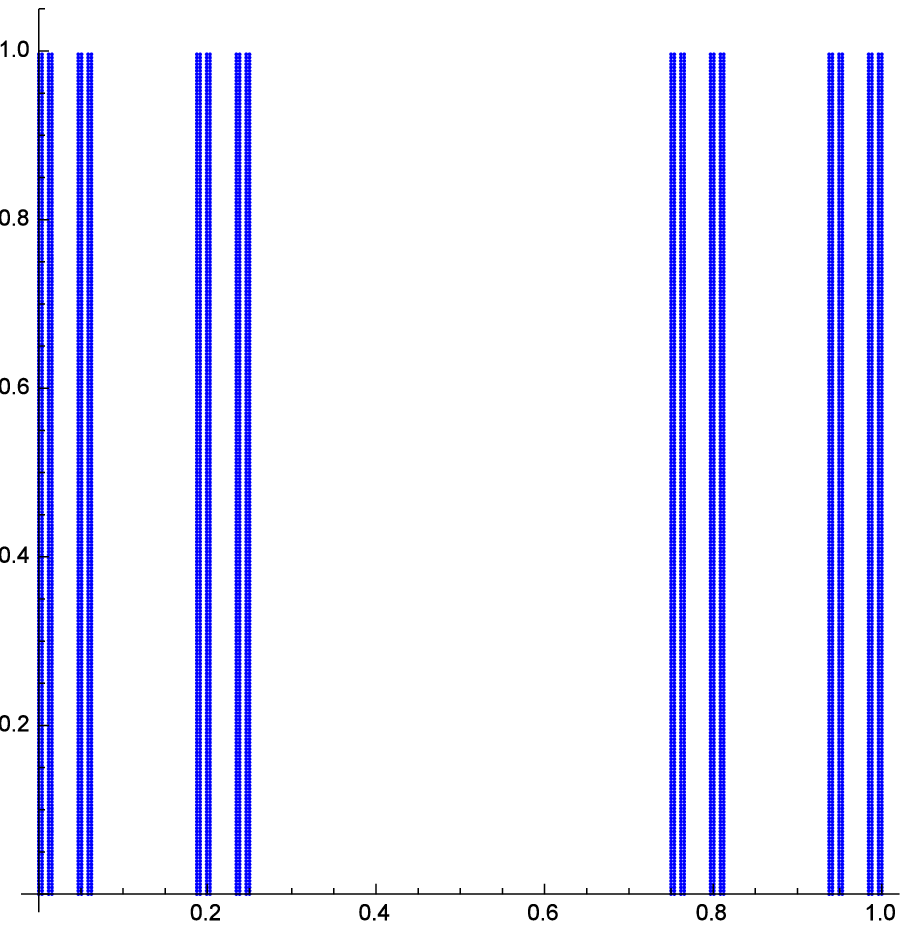}\end{tabular}
\end{center}\vskip -0.5cm
\caption{The IFS $\Fc_1$ and its attractor $E_1$.
}\label{cantor-comb}
\end{figure}
\end{exam}

\begin{exam}\label{comb+}
Let $\Fc_2=\Fc_1\cup\{g,h\}$ be an IFS as illustrated in Figure \ref{film} and $E_2$ the attractor.
It is seen that $x\times [0,1]$ is a component of $E_2$ for every $x\in \Kc$,
and from this one deduces that $E_2$ has a non-locally connected component $Q=h(\mathbb K)$, where $\displaystyle {\mathbb K}=(\Kc\times[0,1])\cup ([0,1]\times \{1\})$
is a typical example of non-locally connected continuum and will be called the \emph{Cantor comb}.
\begin{figure}[ht]
\begin{center}\vskip -0.1cm
\begin{tabular}{cc}
 \includegraphics[width=4.5cm]{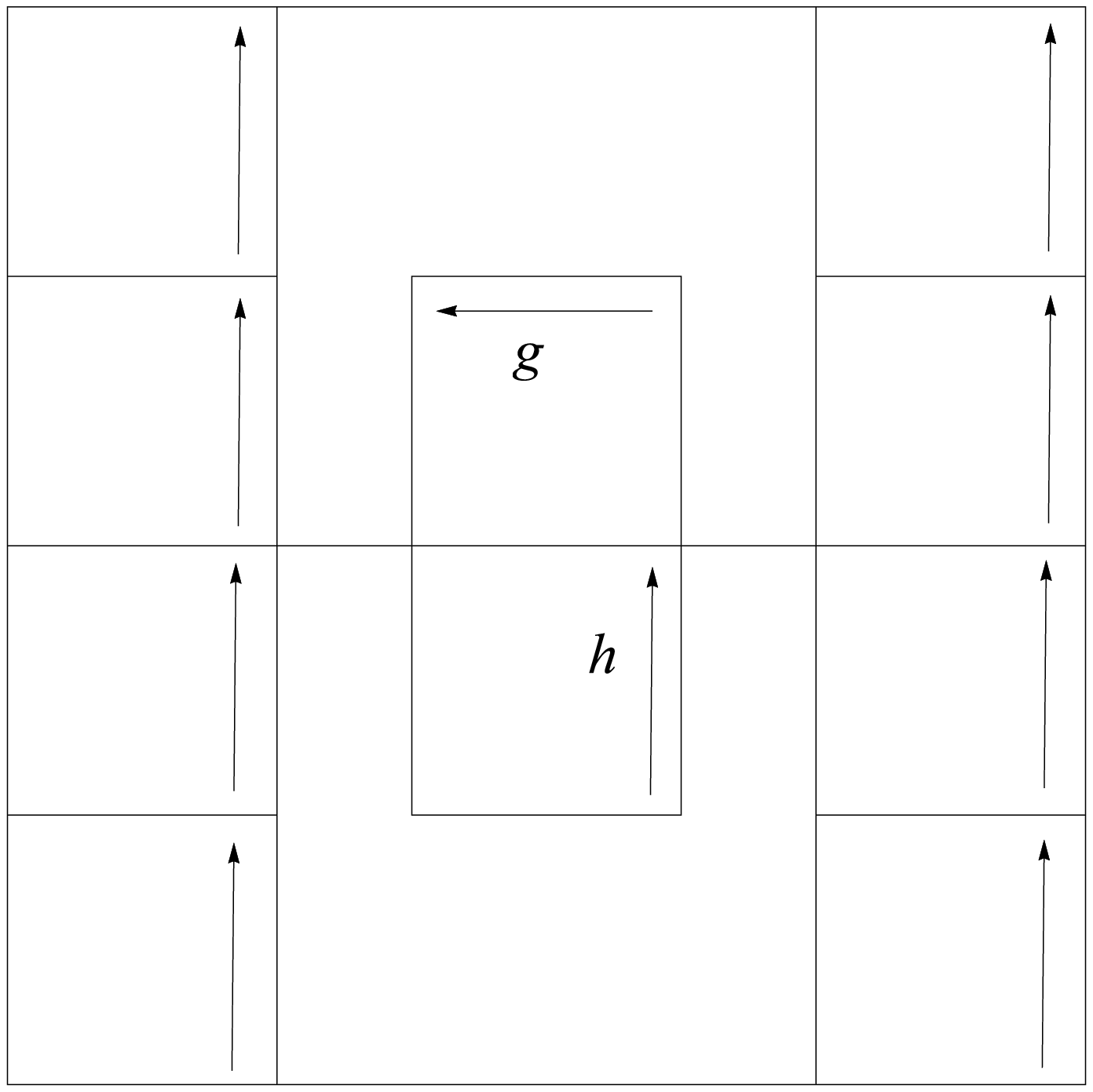}&
\includegraphics[width=4.5cm]{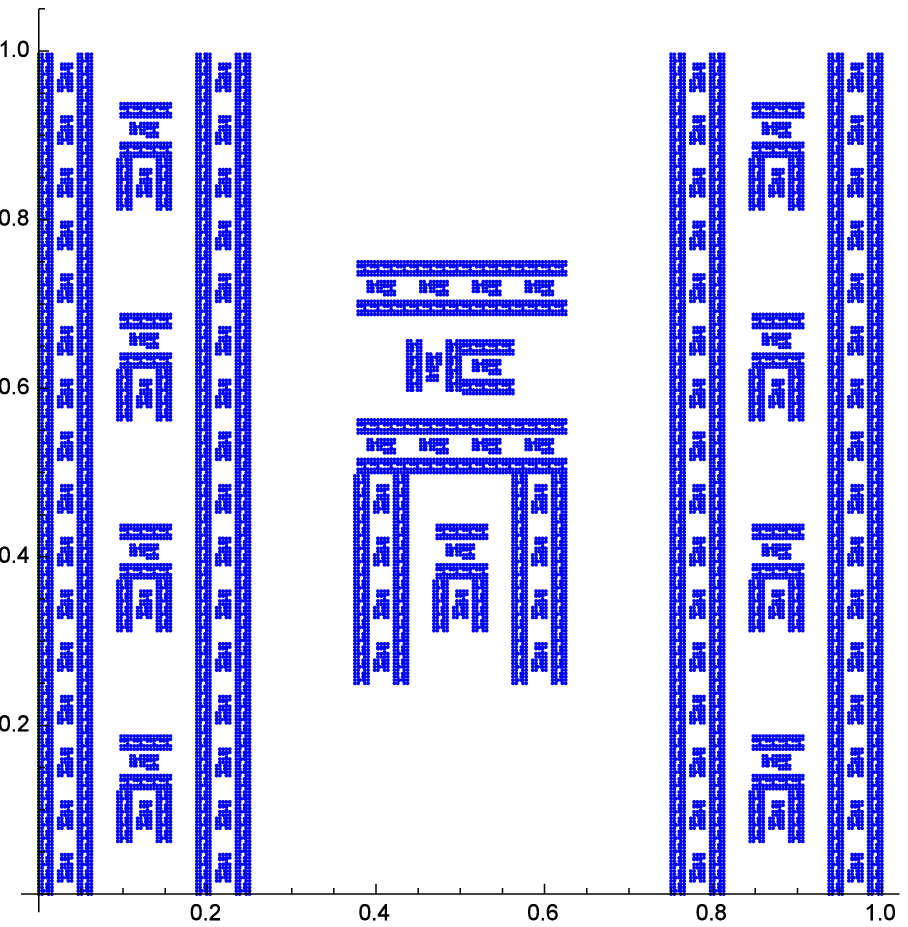}\end{tabular}
\end{center}\vskip -0.5cm
\caption{The IFS $\Fc_2$ and its attractor $E_2$}.
\label{film}
\end{figure}
\end{exam}

\begin{exam}\label{leaves}
Let $\Fc_3$ be the IFS \ illustrated in Figure \ref{leave}, where the $12$ small squares are all of size $1/4$.
Let $E_3$ be the attractor. See Figures \ref{leave} and \ref{leave-2}.  Then the component of $E_3$ containing $2+\frac34\mi$ is not locally connected. We put the proof in the appendix.
\begin{figure}[h]
\begin{center}\vskip -0.1cm
 \includegraphics[width=8cm]{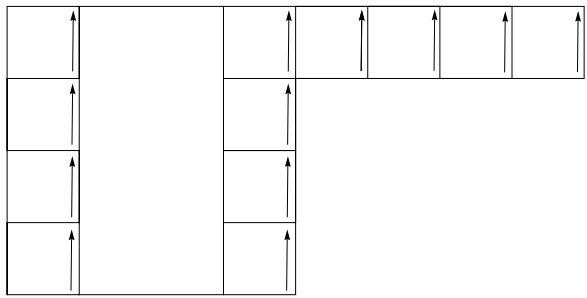}
 \end{center}\vskip -0.5cm
\caption{Images of the unit square under contractions $f\in\Fc_3$.
}\label{leave}
\end{figure}
\begin{figure}[h]
\begin{center}
\includegraphics[width=11cm]{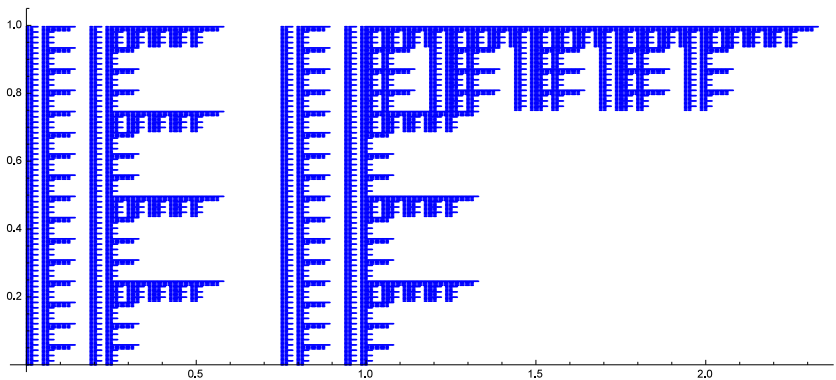}
 \end{center}\vskip -0.5cm
\caption{The self-similar set $E_3$ determined by $\Fc_3$.
}\label{leave-2}
\end{figure}
\end{exam}

From the examples we see that local connectedness is a delicate property for self-similar sets.
Nevertheless, in this paper, we present a positive result concerning fractal squares.

\begin{defi}[cf.~\cite{LLR}]
	Let $n\ge 2$ be an integer, and let  $\Dc$ be a subset of $\{0,1,\ldots,n-1\}^2$ with cardinality $\#\Dc\ge 2$. The attractor of the IFS
	$$\left\{ f_d(x)=\frac{1}{n}(x+d);~d\in \Dc\right \}$$
	is called a \emph{fractal square} and $\Dc$ is called the digit set for $F$.
\end{defi}

Motivated by the study of Lipschitz equivalence of self-similar sets, Xi and Xiong \cite{XX} and Roinestad \cite{Roi-2010} studied the topological  properties of fractal squares; in particular, several criteria of  total disconnectedness are given.
More recently Lau,Luo and Rao \cite{LLR} provide a more complete understanding on the topological structure of fractal squares by investigating $H=F+{\mathbb Z}^2$, a periodic extension of $F$.
Based on the topological classification of fractal squares in \cite{LLR}, we obtain the following result.

\begin{main-theorem}\label{main-3}
Every component of a fractal square is locally connected.
\end{main-theorem}


However, the above result fails in higher dimensions, as indicated by the following example.

\begin{exam}\label{fractal-cube}
Let $\Dc_1=\left\{\left.(i,0,2)\right| i=0,1,2\right\}$ and $\Dc_2=\left\{\left.(i,j, 0)\right| i=0,2; j=0,1,2\right\}$; let $\displaystyle f_d(x)=\frac{x+d}{3}$ for $d\in\Dc=\Dc_1\cup\Dc_2\cup\{(0,2,1)\}$; let $E_4$ be the attractor of the IFS $\mathcal F_4=\{f_d\}_{d\in\mathcal D}$. See Figure \ref{3D}.
We are going to show that a component of $E_4$  is homeomorphic with the Cantor comb mentioned in Example \ref{comb+}. 
\begin{figure}[h]
 \hskip -1.5cm
 \includegraphics[height=6cm]{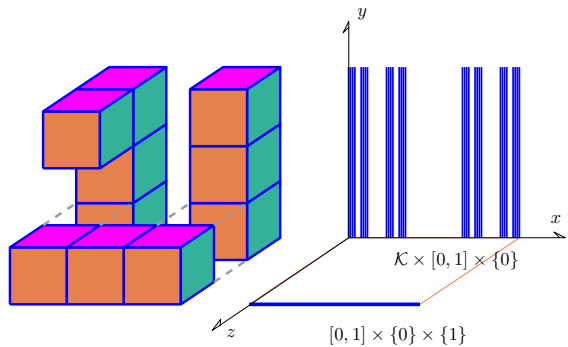}\\
\vskip -1.5cm
\caption{The IFS $\Fc_4$ and two $z$-sections of the attractor $E_4$.}
\label{3D}
\end{figure}

Notice that the projection $(x,y,z)\xrightarrow{P_1} z$ sends $E_4$ to the whole interval $[0,1]$. We will call  $E_4\cap \{z=z_0\}$ a \emph{$z$-section} of $E_4$ for each $z_0\in [0,1]$. Moreover,  the projection $(x,y,z)\xrightarrow{P_2} (y,z)$ maps $E_4$ to a fractal square $F$ on the $yz$-plane. The digit set for $F$ is
$${\mathcal F}=\{(0,0),(1,0),(2,0),(2,1),(0,2)\}.$$
We claim that every component of $F$ is either a single point or a line segment parallel to the $y$ axis. First, notice that  there is a broken line in $(0,1)^2\setminus F$ starting from
the point $(0, 11/18)$  and ending at the point $(1,11/18)$, see
 the dotted line in Figure \ref{p2-cube}.
\begin{figure}[h]
\begin{center}\vskip -0.25cm
 \includegraphics[height=4.5cm]{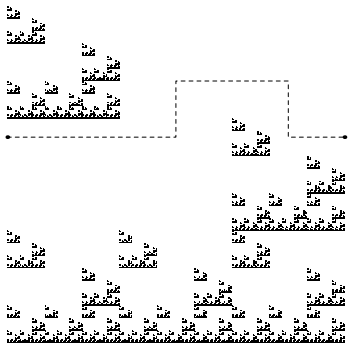}\\
\end{center}
\vskip -1.0cm
\caption{The attractor $F$ on the $yoz$-plane and the (dotted) broken line.}
\label{p2-cube}\vskip -0.25cm
\end{figure}
Let $H=F+\mathbb{Z}^2$,  then
$\mathbb{R}^2\setminus H$ has an unbounded component, for example the component containing
the dotted line in the above figure. Then, by Theorem 1.1 of \cite{LLR},   all nontrivial components of $F$ are line segments parallel to $\mathbb{R}\times\{0\}$. Consequently, every component of $E_4$ must be a subset of a single $z$-section. Since the $z$-sections with $z=1$ and $z=0$ are respectively $P_1^{-1}(1)=[0,1]\times\{0\}\times\{1\}$ and $P_1^{-1}(0)=\Kc\times [0,1]\times\{0\}$ (see Figure \ref{3D}), we see that the $z$-section $P_1^{-1}(1/3)$ is
 $$f_{(0,2,0)}\left(P_1^{-1}(1)\right)\cup f_{(0,2,1)}\left(P_2^{-1}(0)\right).$$
It is connected and hence is a component of $E_4$; moreover, it is homeomorphic to the Cantor comb ${\mathbb K}$ mentioned in Example \ref{comb+}.\qed
\end{exam}

By a classical result of Hahn-Mazurkiewicz (\cite[p.256, $\S$50, II, Theorem 2]{Kur68}), if a component of a self-similar set is locally connected then it must be path connected.
The self-similar sets in Examples 1.1-1.4 share the following property: all the components are path connected.
In Section 5 (Example \ref{non-path}), we give a self-similar set in ${\mathbb R}^3$ such that one of its components is not path connected.

We note that one of the mappings in $\Fc_2$ of Example \ref{comb+}  involves  a rotation, while
$\Fc_3$ in Example \ref{leaves} is an IFS that does not satisfy the open set condition.
(For the open set condition, we refer to Falconer \cite[p.129]{Fal}.)
Encouraged by these examples we state the following conjecture.

\begin{conj}\label{lc-component}
\emph{Let $\Fc$ be an IFS on $\bbR^2$  and $E$ its attractor. If $E$ satisfies the open set condition and if  every contraction in $\Fc$ involves neither rotation nor reflection,  then every component of $E$ is locally connected {\rm(path connected)}.}
\end{conj}

The paper is organized as follows. Section \ref{basics} briefly reviews on topological notions and results related to local connectedness, and proves Theorem \ref{main-1}. Sections \ref{lemmas} and \ref{fractal-square} prepare a few useful lemmas and give the proof for Theorem \ref{main-3}. Section \ref{non-path} considers self-similar sets $E$ in $\bbR^3$.

\section{Preliminary on connectedness and local connectedness}\label{basics}

This section starts from a short review on notions and results of topology, which can be found in the book ``Dynamic Topology'' by Whyburn and Duda \cite{Whyburn-Duda}.


Let $X$ be a (topological) space and $Y$ a subset of $X$.
The {\em induced topology} on $Y$ consists of the sets $Y\cap W$  where $W$ runs over all the open set in  $X$. A decomposition $Y=A\cup B$ into nonempty sets $A,B$ with $A\cap\cl(B)=\cl(A)\cap B=\emptyset$ is called a {\em separation of $Y$}. A subset $Y$ is said to be {\em connected} if it admits no separation whatever; otherwise, it is disconnected.

Let $Y$ be a subset of a space $X$. A {\em component} of $Y$ is a maximal connected set $P\subset Y$
. If $x_0\in P$ we say that $P$ is the component of $Y$ crossing $x_0$. It is known that every component of a set $Y$ is closed in $Y$
\cite[p.35]{Whyburn-Duda}. If the subspace $Y$ is {\em locally connected} then the components of $Y$ are also open in $Y$. A {\em closed-open} set is both closed and open in a topological space; a component of a locally connected space is closed-open. One can easily show that a locally connected compact set $E$ has finitely many components, by using compactness of $E$.

\begin{defi}\label{lc}
A topological space $Y$ is locally connected at a point $x\in Y$ if the component of any open set $U_x\ni x$ crossing $x$ contains an open set $V_x$ with $x\in V_x$. The space $Y$ is locally connected if it is locally connected at all of its points.
\end{defi}

\begin{rem}
The above Definition \ref{lc} of local connectedness is equivalent to the one given in \cite[p.45]{Whyburn-Duda}, which defines a set $M$ (in a space $X$) to be locally connected at a point $p\in M$ provided that if $U_p$ is any open set about $p$, then there exists an open set $V_p$ with $p\in V_p\subset U_p$, such that each point of $M\cap V_p$ lies together with $p$ in a connected subset of $M\cap U_p$.
\end{rem}

\begin{defi}\label{peano}
A compact connected set in a metric space is called a {\em continuum}.  A locally connected continuum in a metric space is called a Peano continuum. A single point is considered as a {\em degenerate} Peano continuum.
\end{defi}

A set $M$ in a metric space is said to have {\em property {\bf S}} if for every $\epsilon>0$ the set $M$ is the union of a finite number of connected sets, each of diameter less than $\epsilon$. A set having property {\bf S} is locally connected \cite[p.48]{Whyburn-Duda}.

\begin{proof}[\textbf{Proof of Theorem \ref{main-1}.}] As we have mentioned before, if $E$ is a locally connected self-similar set, then $E$ has finitely many components since it is compact. It follows that the components of~$E$ are both compact and open. Thus these components are locally connected since so does~$E$. In other words, the components of~$E$ are Peano continua.

Suppose that $E$ is determined by $\{f_1,\ldots,f_q\}$ and that $E$ has finitely many components, say $P_1,\ldots,P_N$. It is easy to check that for any integer $n\ge1$ we have
\[E=\bigcup_{\alpha\in\{1,\ldots,q\}^n} f_\alpha(E),\]  where $\displaystyle f_{\alpha}=f_{j_1}\circ \cdots\circ f_{j_N}$ for any word $\alpha=j_1j_2\cdots j_N$ in $\{1,\ldots,q\}^n$. Let $r$ be the maximal ratio of the contractions $f_i$. Then \[\Uc_n:=\left\{f_\alpha(P_i): \alpha\in\{1,\ldots,q\}^n, 1\le i\le N\right\}\]
is a finite cover of $E$ by connected sets of diameter less than $r^n diam(E)$. 
Since $r^n diam(E)$ approaches $0$ as $n\rightarrow\infty$, the self-similar set $E$  is locally connected.
\end{proof}


%

\section{Two Lemmas}\label{lemmas}

In this section, we build two lemmas concerning components of compact sets on $\bbR^2$. Let us start with   the notion of quasi-component.

\begin{defi}\label{quasi}
The quasi-component of a space $X$ crossing a point $x\in X$ is the intersection of all closed-open sets containing $x$. In other words, it consists of all the points $y\in X$ such that there is no separation $X=A\cup B$ with $x\in A, y\in B$ \cite[p.148, $\S46$, V]{Kur68}.
\end{defi}

Given a space $X$, the component of $X$ crossing $x$ is contained in the quasi-component crossing $x$ \cite[Theorem 1, p.148, $\S46$, V]{Kur68}.  The converse is  not generally true. Actually, let
\[ X=\left(\bigcup_{n\ge1}[-1,1]\times\left\{\frac{1}{n}\right\}\right)\cup\{(-1,0), (1,0)\},\]
then  $(-1,0)$ and $(1,0)$ lie in the same quasi-component, although they belong to different components  \cite[pp.145-146, $\S46$, IV]{Kur68}. However, for a compact space we have:

\begin{theo}{\rm\cite[p.169, $\S47$, II, Theorem 2]{Kur68}}\label{O}
In compact spaces the quasi-components coincide with the components.
\end{theo}

Our first lemma is a pasting lemma.

\begin{lemm}[\textbf{Pasting lemma}]\label{A}
Let $Y\subset\bbR^n(n\ge2)$ be a compact set and $Y_1$ be a component of $Y$, let
 $Z\subset \bbR^n$ be a compact set such that  $Y_1\cap Z=\emptyset$. Then $Y_1$ is also a component of $Y\cup Z$.
\end{lemm}
\begin{proof} We only need to consider the case $Z\cap Y\ne\emptyset$.
By Theorem \ref{O}, the set $Y_1$ is   a quasi-component of $Y$. As $Y\cap Z$ is a compact set disjoint from $Y_1$, we can find a separation $Y=A\cup B$ with $Y_1\subset A$ and $A\cap(Y\cap Z)=\emptyset$. Clearly, $Y\cup Z=A\cup(B\cup Z)$ is a separation. Since $Y_1$ is a quasi-component of $A$, it is also a quasi-component of $Y\cup Z$.
\end{proof}

Our second lemma is based on the following.
\begin{theo}{\rm\cite[Lemma 2.1]{LAT}}\label{brick-wall}
Suppose that $A$, $B$ are disjoint closed sets with
$A \subset  [0,1)\times[0,1]$ and $B\subset (0,1]\times [0,1]$. Then there exists a path in $([0,1]\times[0,1])\setminus(A\cup B)$ starting from a point in $(0,1)\times\{0\}$ and leading to a
point in  $(0,1)\times\{1\}$.
\end{theo}

\begin{lemm}[\textbf{Separation lemma}]\label{B}
Let $A,B\subset[0,1]^2$ be disjoint nonempty compact sets. If a component $P$ of $A$ and a component $Q$ of $B$ both intersect $\partial[0,1]^2$, then we can find an open arc $\gamma$ in $(0,1)^2\setminus(A\cup B)$ 
such that the two components of $[0,1]^2\setminus\cl(\gamma)$
contain $P$ and $Q$ respectively.
\end{lemm}

\begin{proof}
As $P$ and $Q$ are components of $A\cup B$, by a homeomorphism we can transform our situation  in a way that
$\left(P \cap \partial[0,1]^2\right) \subset \{0\} \times (0,1)$ and $\left(Q \cap \partial[0,1]^2\right) \subset \{1\} \times (0,1)$. Fix a separation $A\cup B= A'\cup B'$ with $P\subset A'\subset [0,1)\times[0,1]$ and $Q\subset B'\subset (0,1]\times [0,1]$.  Then we can use Theorem \ref{brick-wall} to get an open arc $\gamma$ connecting a point on $(0,1) \times \{0\}$ to a point on $(0,1) \times \{1\}$ which is outside $A'\cup B'=A\cup B$. This arc $\gamma$ is clearly what we need.
\end{proof}

\section{On components of a fractal square}\label{fractal-square}
In this section, we  prove  Theorem \ref{main-3}.

Let $\Fc_\Dc:=\left\{f_d(x)=\frac{x+d}{n}: d\in\Dc\right\}$ be the IFS determining the fractal square $F$. Let $F_0$ be the unit square, and set $F_{k}=(F_{k-1}+\Dc)/n$ for all $k\ge1$. For each $k\ge1$ and $1\le i,j\le n^k$, a square of the form $\left[\frac{i}{n^k},\frac{i+1}{n^k}\right]\times
\left[\frac{j}{n^k},\frac{j+1}{n^k}\right]$ is called a $k$-cell.
Let $H=F+\bbZ^2$ and $H_k=F_k+\bbZ^2$ be the periodic extension of $F$ and $F_k$, respectively. We have
\begin{lemm}[cf.~{\cite[Equations (2.1) and (2.2)]{LLR}}]
$H_{k+1}\subset H_k/n$, $H\subset H/n$, $H_k^c/n\subset H_{k+1}^c$ and $H^c/n\subset H^c$ hold for $k\ge1$.
\end{lemm}
\begin{lemm}[cf.~{\cite[Theorem 2.2]{LLR}}]\label{basic-fact}
Every component of $H^c:=\mathbb{R}^2\setminus H$ is unbounded if one of its  component is of diameter greater than  $\sqrt{2}(n^2+1)^2/n$, and in this case each component of $F$ is either a line segment, or a single point.
\end{lemm}

\begin{lemm}\label{position}
If $\gamma$ is an open arc  in the interior of $F_N\setminus F$ for some $N\ge2$ then $n^N\gamma \subset H^c$.
\end{lemm}
\begin{proof}
Let $\mathcal{D}_0=\mathcal{D}$ and $\mathcal{D}_{k}=\mathcal{D}_0+n\mathcal{D}_{k-1}$ for $k\ge1$. Then, by repeatedly using $nF=F+{\mathcal D}$ (implying $n^NF=F+\mathcal{D}_N$) and $nF_{k}=F_{k-1}+{\mathcal D}$, we arrive at
$$
n^N\gamma \subset \left ([0,1]^2+{\mathcal D}_N\right )^o \text{ \ and \ } n^N\gamma\cap (F+\mathcal{D}_N)=\emptyset.
$$
If $n^N\gamma$ intersects $z+F$ for some $z\in {\mathbb Z}^2$ the first relation implies that $z\in D_N$, contradicting the second relation.
\end{proof}

The following two results are used when we prove Theorem \ref{main-3}.
\begin{lemm}[cf.~{\cite[Part B, Section VI, Torhorst Theorem and Lemma 2]{Whyburn-Duda}}]\label{torhorst}
The boundary of each component of the complement of a locally connected continuum $M$ is itself a locally connected continuum. Moreover, if $M$ has no cut point then every component of $\mathbb{R}^2\setminus M$ is bounded by a simple closed curve.
\end{lemm}

\begin{lemm}[cf.~{\cite[p.108, (3.1) Separation Theorem]{Whyburn42}}]\label{separation-theorem}
If $A,B\subset\mathbb{R}^2$ are compact sets such that $A\cap B$ is totally disconnected and $a,b$ are points of $A\setminus B$ and $B\setminus A$, respectively, and $\epsilon$ is any positive number, then there esists a simple closed curve $J$ which separates $a$ and $b$ and is such that $(J\cap(A\cup B))\subset (A\cap B)$, and every point of $J$ is at a distance less than $\epsilon$ from some point of $A$.
\end{lemm}

\begin{rem}\label{separation-case}
We will apply the above lemma to a particular case, when $A$ is a component of a compact set $K\subset\mathbb{R}^2$ and $B\subset K$ compact with $A\cap B=\emptyset$. In such a case, by Theorem \ref{O} we can find for any $\epsilon>0$ a separation $K=E\cup F$ with $A\subset E\subset A_\epsilon$ and $B\subset F$. Here $A_\epsilon$ is the $\epsilon$-neighborhood of $A$ and consists of all the points at a distance less than $\epsilon$ to $A$. Then, by Lemma \ref{separation-theorem} there is a simple closed curve $\Gamma$ with $\Gamma\cap (E\cup F)=\emptyset$ such that $A$ lies in one component of $\mathbb{R}^2\setminus \Gamma$ and $B$ in the other and that every point of $\Gamma$ is at a distance less than $2\epsilon$ from some point of $A$.
\end{rem}

\begin{proof}[{\bf Proof for Theorem \ref{main-3}.}]

Assume on the contrary that $F$ has a component $X$ which is not locally connected at a point $x_0\in X$.
We will show that $H^c$ would have a component of diameter $>\sqrt{2}(n^2+1)^2/n$. By  Lemma \ref{basic-fact} this implies that every component of $F$ is either a line segment or a single point, thus provides a contradiction and ends our proof.

By the above assumption, we can find an open set $V_0\ni x_0$ such that the component of $V_0\cap X$ containing $x_0$, which we denote by $P'$,  does not contain any neighborhood of $x_0$,  under the induced topology of $X$.

Let $N_0$ be the smallest integer  such that all the $N_0$-cells containing $x_0$ is a subset of $V_0$.
There are at most $4$ such cells, and we denote their union by  $V$; then $V$ is a rectangle lying inside $V_0$,  and  $x_0$ belongs to the interior of $V$.
Let $\displaystyle \varepsilon_0=\frac{1}{5}{\rm dist}(x_0,\partial V). $
We consider the components of the compact set $F\cap V$ and let $P_0$ be the component of $F\cap V$ containing $x_0$. It is standard to check that all components of $X\cap V$ are each a component of $F\cap V$. As $P_0\subset P'$, the component $P_0$ is not a neighborhood of $x_0$ under the induced topology of $X$.

The rest of our proof is divided into four lemmas. The basic idea is to find an integer $N$ satisfying $n^N\varepsilon_0>\sqrt{2}(n^2+1)^2/n$ and an open arc $\gamma\subset (F_N\setminus F)$ of diameter greater than $2\varepsilon_0$. See Lemma \ref{good-gamma}. Then, by Lemma \ref{position} we have an arc $n^N\gamma\subset H^c$ and a component of $H^c$, the one containing $n^N\gamma$, both of which are of diameter greater than $\sqrt{2}(n^2+1)^2/n$. By Lemma  \ref{basic-fact} this completes the whole proof.

\begin{lemm}\label{sequence-of-components}
We can find a sequence $\{P_k: k\ge1\}$ of components of $X\cap V$ and points $x_k, y_k\in P_k$ with $y_k\in\partial V$ and $|x_k-y_k|>4\varepsilon_0$ such that $\lim\limits_{k\rightarrow\infty}x_k=x_0$, $\lim\limits_{k\rightarrow\infty}y_k=y_0$, and $\lim\limits_{k\rightarrow\infty}P_k=P_\infty$ under Hausdorff distance. Here  $\displaystyle \varepsilon_0=\frac{1}{5}{\rm dist}(x_0,\partial V). $
\end{lemm}

Clearly, we can find a point $x_1$ in $(X\cap V)\setminus P_0$  with $|x_1-x_0|<\varepsilon_0$. Let $P_1$ be the component of $X\cap V$ containing $x_1$. Since $P_1$ is closed and $\text{dist}(x_0,P_1)>0$, we can find a second point $x_2$ in $(X\cap V)\setminus(P_0\cup P_1)$  such that $|x_2-x_0|<\varepsilon_0/2$. Denote by $P_2$ the
component of $X\cap V$ containing $x_2$. Continue this procedure,
we can find an infinite sequence of distinct components $\{P_k: k\ge1\}$ of $X\cap V$ and points $x_k\in P_k$ such that $|x_k-x_0|<\varepsilon_0/k$.
Here we note that every $P_k$ intersects the boundary $\partial V$ of $V$. Actually, if on the contrary $P_k\cap \partial V=\emptyset$ for some $k$ then $P_k$ would be a component of $X$ according to Lemma~\ref{A}, violating the connectedness of $X$. Therefore, for every $k\ge1$, we can choose a point $y_k$ in $P_k\cap \partial V$. Here, we have $|x_k-y_k|>4\varepsilon_0$ for all but finitely many $k\ge1$.
By choosing an appropriate subsequence, we shall obtain the result of Lemma \ref{sequence-of-components}. See the following Figure \ref{V} for a simplified depiction.
\begin{figure}[ht]
\centering
\vskip -0.1cm
 \includegraphics[height=5cm]{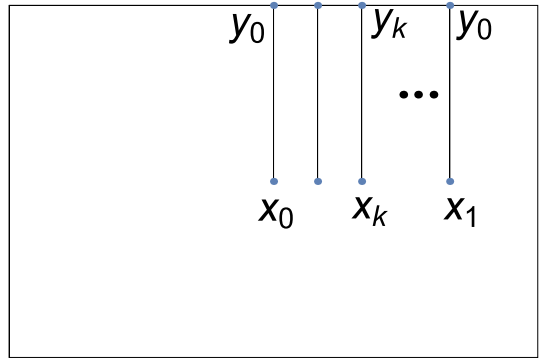}\\
\vskip -0.5cm
\caption{Relative locations of $x_k$ and $y_k$ in the rectangle $V$.}
\label{V}
\end{figure}

\begin{lemm}\label{integer-N}
Pick a large enough integer $N>2N_0$ such that
$n^N\varepsilon_0>\sqrt{2}(n^2+1)^2/n$. Then each $N-$cell $\left[\frac{i}{n^N},\frac{i+1}{n^N}\right]\times
\left[\frac{j}{n^N},\frac{j+1}{n^N}\right](1\le i,j\le n^N)$ is of diameter
$\displaystyle \sqrt{2}n^{-N}<\frac{n\varepsilon_0}{(n^2+1)^2}<
\frac{\varepsilon_0}{5}$.
\end{lemm}

The above lemma is immediate. And, for the sake of convenience, an $N-$cell is called a \textbf{black $N-$cell} if it has the form $\displaystyle   f_\alpha([0,1]^2)$ for some $\alpha\in \{1,\dots,q\}^N$
and  intersects infinitely many $P_k$; otherwise, it is called a \textbf{white $N-$cell}.
Let $M$ be the union of the black cells. Then all but finitely many $P_k$, and also $P_\infty$, is contained in $M$, which in turn is contained in $F_N$.  Clearly, $M^o$ has finitely many components $W_1,\ldots,W_p$; moreover, the closures of any two components intersect at a finite set. Since for each $W_i$ the closure $\cl(W_i)$ intersects infinitely many $P_k$, we have $P_\infty\cap\cl(W_i)\ne\emptyset$ for $1\le i\le p$. Moreover, there are at most finitely many $P_k$ intersects the finite set $\cl(W_i)\cap\left(\bigcup_{j\ne i}\cl(W_j)\right)$. Since the components $P_k$ are disjoint continua, we can infer that every $\cl(W_i)$ contains all but finitely many of the components $P_k$ with $P_k\cap\cl(W_i)\ne\emptyset$. Therefore, we have $P_\infty\subset\cl(W_i)$ and $p=1$; in other words, the interior $M^o$ of $M$ is connected.
Consequently, $M$ has no cut point. By Lemma \ref{torhorst}, the unbounded component $W$ of $\bbR^2\setminus M$ is bounded by a simple closed curve $J$.

\begin{lemm}\label{M-star}
There is a continuum $M^*\subset M$ such that:
(1) the unbounded component of $\mathbb{R}^2\setminus M^*$ is bounded by a simple closed curve $J^*$, (2) there is an integer $k_0>1$ such that $M^*$ contains $\bigcup_{k\ge k_0}P_k$ and thus contains $P_\infty$, and (3) every bounded component $U$ of $Int(J^*)\setminus M^*$ satisfies $\cl(U)\cap\left(\bigcup_{k\ge k_0}P_k\right)=\emptyset$.
\end{lemm}

Let $U_1,\ldots, U_p$ be the bounded components of $\bbR^2\setminus M$, called `holes' of $M$. Then
$J\cup Int(J)$ $=M\cup U_1\cup\cdots\cup U_p$.
Suppose that the closure $\cl(U_i)$ of a hole $U_i$ intersects some $P_k$. Pick a separation $M\cap F=A\cup B$ with $P_k\subset A$ and $P_\infty\subset B$. By Lemma \ref{separation-theorem} and Remark \ref{separation-case}, we may choose a simple closed curve $\Gamma_1$ satisfying:
(1) $\Gamma_1\cap M\cap F=\emptyset$;
(2) every point of $\Gamma_1$ is at a distance to $P_k$ less than
\[\displaystyle \frac{1}{2}\min\left\{|x-y|: x\in P_k, y\in P_0\cup\left(\bigcup_nP_n\right)\right\};
\]
(3) $P_k$ and $P_\infty$ are contained in distinct components of $\mathbb{R}^2\setminus\Gamma_1$.
Fix an arc $\alpha_1\subset\Gamma_1$ starting from a point on $J$ and ending at a point on $\bigcup_j\partial U_j$, such that the interior of $\alpha_1$ is disjoint from $J\cup\left(\bigcup_k\cl(U_j)\right)$. Assume that the ending point of $\alpha_1$ is on $\partial U_j$. Here one may expect that $j=i$. However, it is possible that $j\ne i$, if $\Gamma\cap\partial U_j\ne\emptyset$. Renaming the indices $1\le i\le p$, if necessary, we may assume that $j=1$.
Thicken $\alpha_1$ to be a closed topological disk  $D_1$ with $\alpha_1\subset D_1\subset (M\setminus F)$ such that $\partial D_1\setminus(\partial M)$ consists of two open arcs, denoted as $\beta_1,\beta_2$. Let $M_1$ consist of all those points $z$ in $\cl(M\setminus D_1)$ such that no finite set $C$ separates $z$ from $x_0$, in the sense that there is a separation $\cl(M\setminus D_1)\setminus C=A\cup B$ with $z\in A$ and $x_0\in B$. If $J\cap(\partial U_j)\ne\emptyset$ for some $j$, say $j=2$, then it is possible that $\cl(M\setminus D_1)$ has a finite cutting set. In such a case, the interior $(M\setminus D_1)^o$ is not connected.
See the following Figure \ref{M1} for relative locations of $U_1, \alpha_1, \beta_1$ and $\beta_2$ in $M$.

\begin{figure}[ht]
\centering
\vskip -0.05cm
 \begin{tabular}{ccc}\includegraphics[height=5cm]{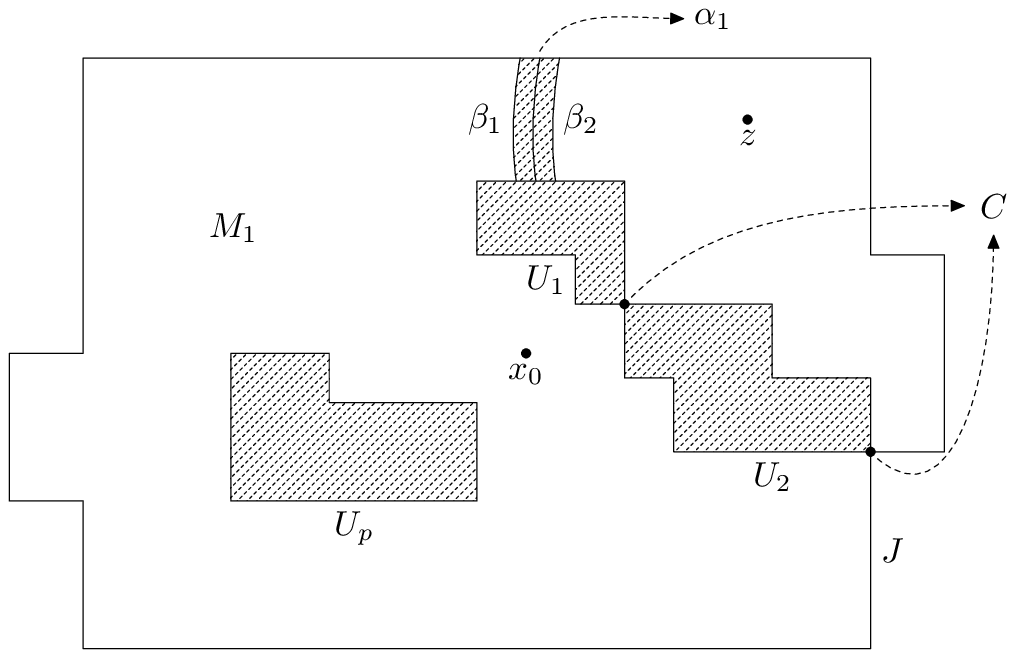} &\hspace{0.25cm} & \includegraphics[height=5cm]{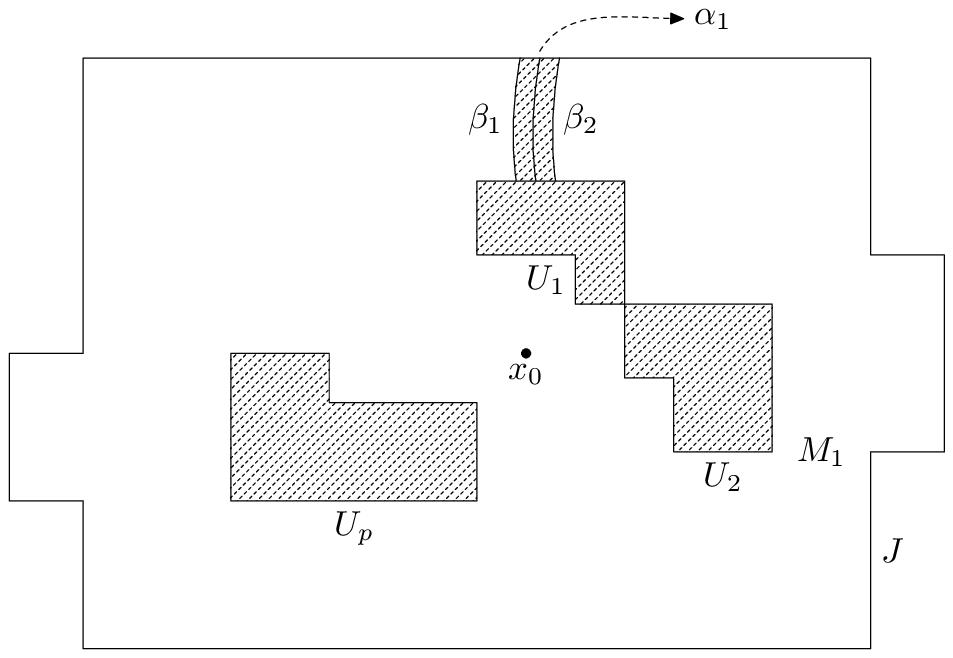} \end{tabular}
\vskip -0.25cm
\caption{ Left: $(M\setminus D_1)^o$ is disconnected. Right: $(M\setminus D_1)^o$ is connected.}
\label{M1}
\vskip -0.05cm
\end{figure}

Now it is direct to check  (1) that each of $\beta_1$ and $\beta_2$ is located on exactly one component of the interior $M_1^o$ and (2) that all but finitely many $P_k$ are contained in $M_1$. From the latter we can infer $P_\infty\subset M_1$.  Therefore, given the common part of the closures for any two components of $M_1^o$, denoted $Z$, we have $Z\cap\left(\beta_1\cup\beta_2\right)=\emptyset$. Since $M$ is a finite union of black cells in $F_N$, we see that $Z$ is necessarily a finite set. Form this we can infer the connectedness of $M_1^o$. Moreover, it is routine to check that $\mathbb{R}^2\setminus M_1$ has no more than $p-1$ bounded components, each of which equals $U_i$ for some $i\ne 1$. Note that the unbounded component $W_1$ of $\bbR^2\setminus M_1$ is bounded by a simple closed curve $J_1$. Clearly, the union $\beta_1\cup\beta_2$ is contained in $J_1$ and does not intersect $F$. If $\partial U_i$ intersects some $P_k$ then we can find a closed arc $\alpha_2$ disjoint from $\beta_1\cup\beta_2\cup(M_1\cap F)$ that starts from a point on $J_1$ and ending at a point on $\partial U_i$. So we can use the above procedure and obtain $M_2$, satisfying the conditions given in Lemma \ref{M-star}. Repeating this procedure for at most finitely many steps, we will obtain $M^*$ as required.

\begin{lemm}\label{good-gamma}
There is an arc $\gamma$ in $(M^*\setminus F)\subset(F_N\setminus F)$ that is of diameter $\ge n^{-N}\!\sqrt{2}(n^2+1)^2/n$.
\end{lemm}
Let $U_1^*,\ldots, U_q^*$ be the bounded components of $\bbR^2\setminus M^*$. Then $Int(J^*)\setminus M^*=U_1^*\cup\cdots\cup U_q^*$. Moreover, the closure $\cl(U_i^*)$ is disjoint from $\bigcup_{k\ge k_0}P_k$ for $1\le i\le q$. By Lemma \ref{A}, every  $P_k$ with $k\ge k_0$ is also a component of
$\displaystyle (F\cap M^*)\cup\cl(U_1^*)\cup\cdots\cup\cl(U_q^*)$. Recall that $D:=J^*\cup Int(J^*)$ is a topological disk. For the previously fixed constant $\varepsilon_0>0$, choose a component $P_k$ with $|x_k-x_0|$ and $|y_k-y_0|$ small enough such that $x_k$ and $y_k$ are respectively connected to $x_0$ and $y_0$ by open arcs $\gamma_1, \gamma_2\subset Int(J^*)$, with diameters smaller than $\varepsilon_0$. By Theorem \ref{O}, there is a separation $\displaystyle (F\cap M^*)\cup\cl(U_1^*)\cup\cdots\cup\cl(U_q^*)=A\cup B$
with $P_k\subset A$ and  $\left[P_\infty\cup\cl(U_1^*)\cup\cdots\cup\cl(U_q^*)\right]\subset B$. Finally, consider $D$ as a homeomorphic image of $[0,1]^2$. By Lemma \ref{B}, we can find an open arc $\gamma$ in $D^o\setminus(A\cup B)$ such that $\cl(\gamma)$ separates $P_k$ from $P_\infty$ in $D$.
Since $\gamma$ must intersect each of $\gamma_1$ and $\gamma_2$, we have
${\rm diam}(\gamma)\ge |x_0-y_0|-{\rm diam}(\gamma_1)-{\rm diam}(\gamma_2)>2\varepsilon_0$ which is greater than $n^{-N}\!\sqrt{2}(n^2+1)^2/n$, by the choice of $N$ in Lemma \ref{integer-N}.
\end{proof}

\section{A self-similar set that has a non-path-connected component}\label{non-path}

This section gives a self-similar set in $\mathbb{R}^3$ that has a non-path-connected component.

Let $E$ and $F$ be two subset of ${\mathbb R}$ defined by
 $$E=\{0\}\cup \bigcup_{k=1}^\infty \left[2^{-2k+1},2^{-2k+2}\right],\quad
 F=\{0\}\cup \bigcup_{k=1}^\infty \left[2^{-4k+2},2^{-4k+3}\right].$$
Then
\begin{align}
[0,1]&=E\cup\frac12E,  & E&=\frac{1}{4}E\cup\left[\frac12,1\right].\\
\left[0,\frac12\right]&=F\cup\frac12F\cup\frac14F\cup\frac18F,
& F&=\frac{1}{16}F\cup\left[\frac14,\frac12\right].
\end{align}
It follows that $[0,1]=E\cup F\cup \frac14F$.  Moreover, $E$ and $F$ are self-similar sets satisfying the open set condition, namely,
\begin{equation}\label{ifs-E}
E=\frac{E}{4}\cup \left (\frac{E}{4}+\frac12 \right ) \cup \left (\frac{E}{2}+\frac12 \right ),
\end{equation}
\begin{equation}\label{ifs-F}
F=\frac{F}{16}\cup \left[\left(\frac12F\cup\frac14F\cup\frac18F\cup\frac{1}{16}F\right)+\frac{1}{4}\right].
\end{equation}

\begin{lemm} $X=E\times[0,1]$ {\rm(green part in Figure \ref{X^*})} is a self-similar set with open set condition.
\end{lemm}\label{X-IFS}
\begin{proof} The IFS of $E$ is $\displaystyle \left\{ \frac{x}{4},\frac{x+2}{4}, \frac{x+1}{2}\right\}$. We extend $x\mapsto \frac{x}{4}$ to the following $4$ maps:
 $$
 (x,y)\mapsto (x,y+j)/{4}, \quad j=0,1,2,3;
 $$
 extend $x\mapsto \frac{x}{4}+\frac12$ to
 $$
 (x,y)\mapsto {(x+2,y+j)}/{4}, \quad j=0,1,2,3;
 $$
 extend $x\mapsto \frac{x+1}{2}$ to
 $$
 (x,y)\mapsto {(x+1,y+j)}/{2}, \quad j=0,1.
 $$
 It is easy to check that these extended maps form an IFS of $X$. Since the similarity dimension of
 this IFS is $2$ and $X$ has non-empty interior, the IFS satisfies the open set condition by a result of Schief \cite{Schief-1994}.
 \end{proof}

Let
 $X^*=X\cup\left(F\times\{1\}\right)\cup \left(\frac14F\times\{0\}\right)$.
  Note that $X^*$ is a continuum having two path components, one is $\{0\}\times [0,1]$, and the other one is the rest of $X^*$. See Figure \ref{X^*}.
We shall construct a self-similar set $G\subset\mathbb R^3$ which contains a component homeomorphic with $X^*$. To this end, we shall construct an IFS consisting of three parts:  the first part generates copies of $F$ and $F/4$ respectively,  the second part generates
a copy of $X$, and the third part put them together properly.
\begin{figure}[ht]
  \begin{center} \vskip -0.25cm
  \includegraphics[width=0.45\textwidth]{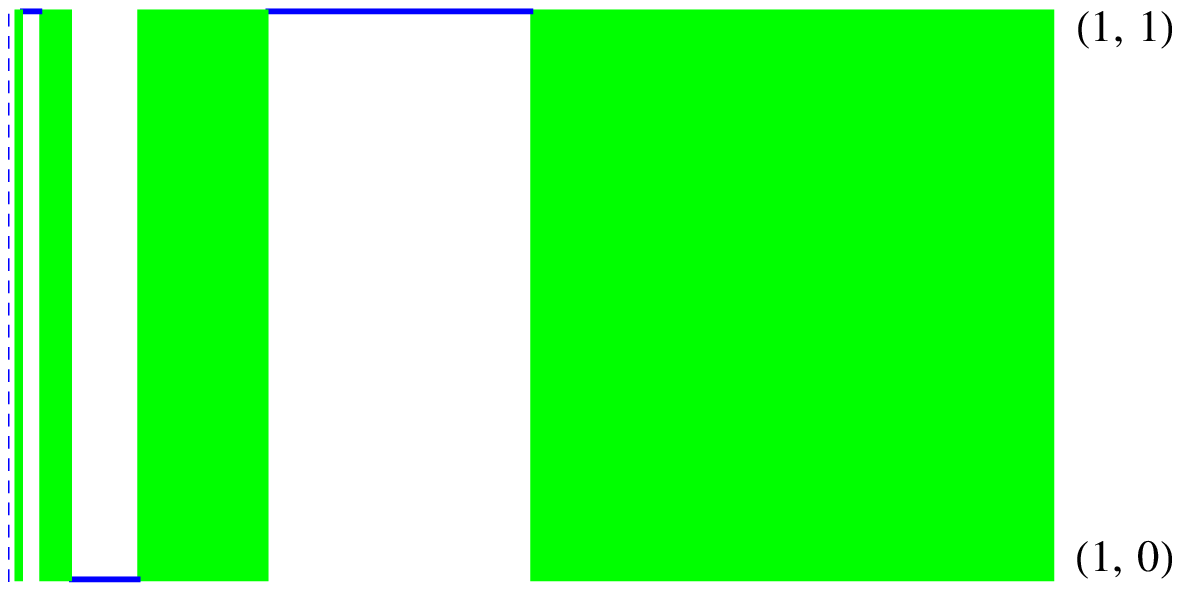}
  \vskip -0.5cm
  \caption{$X^*=X\cup\left(F\times\{1\}\right)\cup \left(\frac14F\times\{0\}\right)$.}
  \label{X^*}
  \end{center}\vskip -0.5cm
\end{figure}
Let ${\cal F}_1$ be an IFS generating $F$, and ${\cal F}_2$ an IFS generating $F/4$ such that the maximal contraction ratios of them are less than $1/5$, which can be obtained by iterating the IFS in \eqref{ifs-F}. Next, for a map $x\mapsto \lambda x+a$ of ${\cal F}_1$, we extend it to a $3$-dimensional map
$$(x,y,z)\mapsto \lambda(x,y,z)+(a,1-\lambda,1-\lambda);
$$
 for a map $x\mapsto \lambda x+a$ of ${\cal F}_2$, we extend it to a map
$$
(x,y,z)\mapsto \lambda(x,y,z)+(a,0,1-\lambda).
$$
Let us denote the extended IFS by ${\mathcal F}_1^*$ and
${\mathcal F}_2^*$, respectively. Let ${\mathcal F}^*={\mathcal F}_1^*\cup {\mathcal F}_2^*$, which is the first part of our IFS.
The following lemma is obvious.

\begin{lemm} The attractor of ${\mathcal F}^*$, denote by $H$, is a subset of $[0,1]^2\times\{1\}$.
Moreover,
$$
H\cap {\mathbb R}\times \{1\} \times \{1\}=F\times \{1\}\times\{1\},
$$
$$
H\cap {\mathbb R}\times \{0\} \times \{1\}=F/4\times \{0\}\times\{1\}.
$$
Furthermore, every component of $F\times \{1\}\times \{1\}$
or $F/4\times \{0\}\times \{1\}$ is also a component of $H$.
\end{lemm}

Iterating the IFS for $X$ in Lemma \ref{X-IFS}, we obtain an IFS  of $X$, denoted as ${\mathcal G}$, such that the maximal contraction ration is less than $1/5$.
Extend each map $(x,y)\mapsto \lambda(x,y)+(a,b)$ in ${\mathcal G}$ to a map
$(x,y,z)\mapsto \lambda(x,y,z)+(a,b,0)$. Then $X\times \{0\}$ is the attractor of the resulted IFS, which we denote  by ${\mathcal G}^*$. This is the second part of our IFS.
Finally, let
$$
h_1(x,y,z)=\frac15(x,y,z)+\left(\frac25,\frac25,\frac25\right),
$$
$$
h_2(x,y,z)=\frac15(x,y,z)+\left(\frac25,\frac45,\frac25\right),
$$
$$
h_3(x,y,z)=\frac15(x,y,z)+\left(\frac25,\frac35,\frac35\right),
$$
and set ${\mathcal F}={\mathcal F}^*\cup {\mathcal G}^*\cup \{h_1,h_2,h_3\}.$
Let $G$ be the attractor of this IFS.

As in Example \ref{fractal-cube}, we call the intersection $G\cap \{(x,y,z): \  z=z_0\}$ the $z_0$-section of $G$. See Figure \ref{z-section} for two of the $z$-sections.
\begin{figure}[H]
  \begin{center}\begin{tabular}{ll}
  \includegraphics[width=0.3725\textwidth]{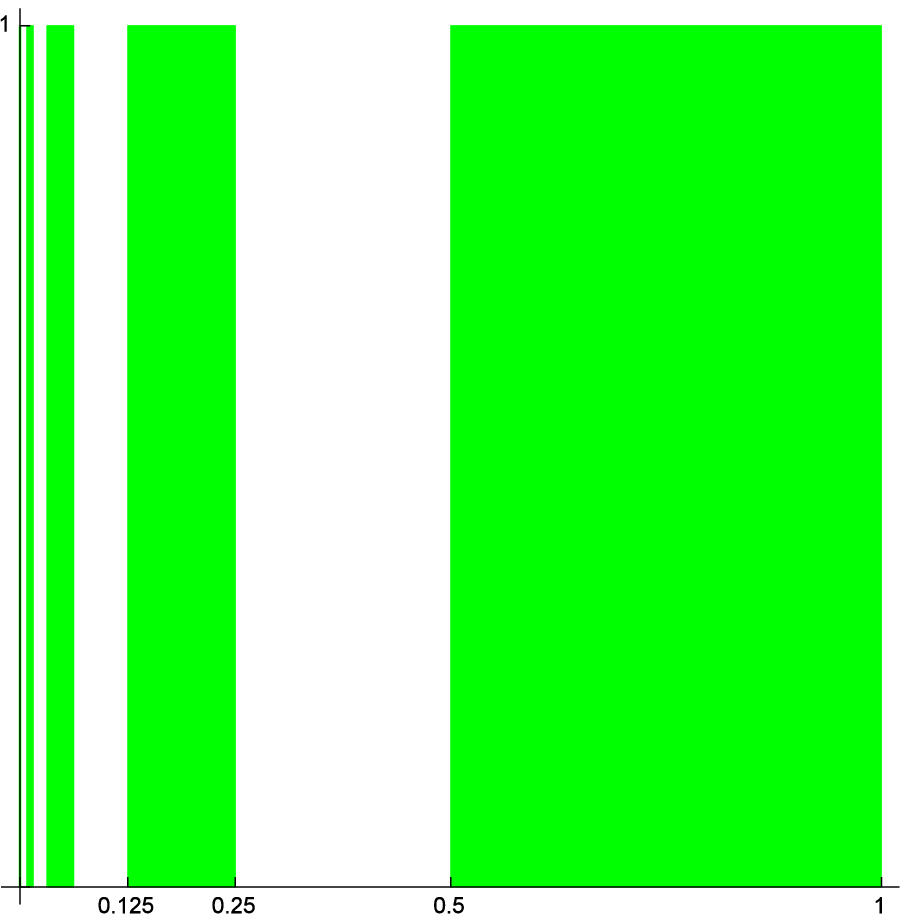}&
    \includegraphics[width=0.20\textwidth]{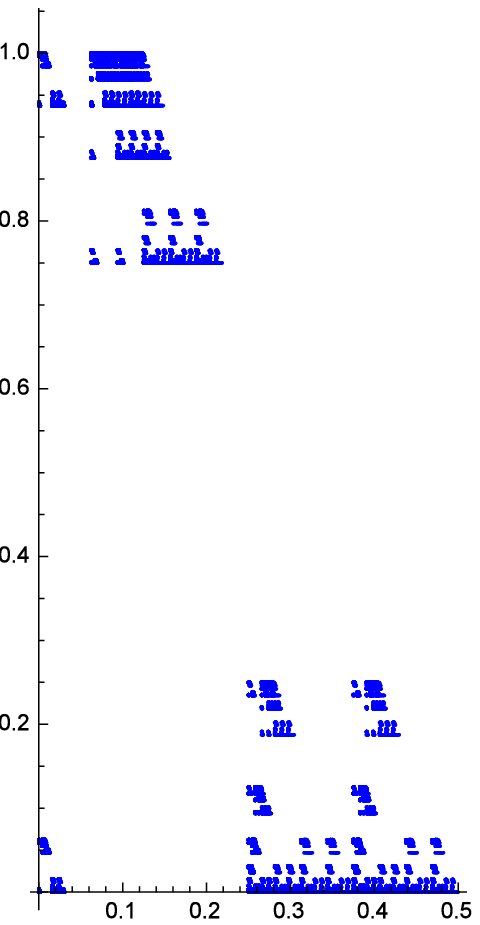}
  \end{tabular}
  \vskip 0cm
  \caption{The two $z$-sections of $G$ at $z=0,1$.}
  \label{z-section}
  \end{center}
\end{figure}

\begin{lemm} (i) Each $f(x,y,z)=\lambda(x,y,z)+(a,b,c)\in {\mathcal F}$ maps $[0,1]^3$ into itself and satisfies $c\notin(3/5,4/5)$.

(ii) The $0$-section of $G$ is $X\times \{0\}$, the $1$-section of $G$ is $H$.

(iii) Under the projection $(x,y,z)\mapsto z$, the image of $G$ is a Cantor set (and hence totally disconnected). Therefore,  a component
of $G$ must belong to a single section.
\end{lemm}

\begin{proof} (i) By our construction of ${\mathcal F}$.

(ii) Let $f=\lambda(x,y,z)+(a,b,c)\in {\mathcal F}$. Since $\lambda\leq 1/5$, $f(G)$ does not intersect the $0$-section if $c\geq 1/5$; and those $f$ must belong to ${\mathcal G}$. It follows that the $0$-section is the attractor of ${\mathcal G}$. The second assertion can be proved in the same way.

(iii) This follows from the fact that $c\leq 3/5$ or $c\ge 4/5$.
\end{proof}

Now we can check that
$$h_3(0\text{-section})\cup h_1(G\cap \mathbb R\times \{1\}\times \{1\} )\cup h_2(\mathbb R\times \{0\}\times \{1\})$$
equals the $3/5$-section of $G$. Moreover, every component of this section is also a component of $G$, one of which is a translation of $\displaystyle \frac{X^*\times\{0\}}{5}:=\left\{\frac{1}{5}(x,y,0): (x,y)\in X^*\right\}$.

\appendix{\noindent \textbf{Appendix A. The Proof of Examples 1.3}}\label{non-lc-component}

\setcounter{section}{0}\stepcounter{section}\def\thesection{\Alph{section}}

Recall that $\Fc_3$ in Example 1.3 actually equals $\displaystyle\left\{  \frac{x+d_k}{4};~\ 1\le k\le 12\right\}$,  where
\[\displaystyle \{d_1,\ldots,d_{12}\}=\left \{  0,\mi, 2\mi, 3\mi, 3, 3+\mi, 3+2\mi, 3+3\mi, 4+3\mi, 5+3\mi, 6+3\mi, 7+3\mi \right\}.\]
Here $\mi$ denotes the imaginary unit with $\mi^2=-1$. In particular, we write
\[f_4(x)= \frac{x+3\mi}{4},\quad f_8(x)=\frac{x+3+3\mi}{4},\quad \quad f_9(x)=\frac{x+4+3\mi}{4},
\quad \quad f_{12}(x)=\frac{x+7+3\mi}{4}.\]
Then $f_8\circ f_9=f_9\circ f_4$, indicating that $\Fc_3$ does not satisfy the open set condition. Moreover, we can obtain the following lemma by routine computations.

\begin{lemm}\label{E_3-a}
Let $Q$ be the closed trapezoid whose vertices are $0, 1, \mi$ and $\frac73+\mi$.
Let $A$ be the line segment $\{t+\mi: 1\le t\le\frac73\}$ and $B=\{\frac58+t\mi: 0\le t\le1\}$. Then we have $A\subset E_3\subset Q$ and $B\cap E_3=\emptyset$. \hfill{\rm (We note that $\frac58$ is the average of $\frac{7}{12}$ and $\frac34$.)}
\end{lemm}

By Lemma \ref{E_3-a} and the definition of $\Fc_3$, we may check that the segments $B_k=\{1-4^{-k}\times\frac38+t\mi:0\le t\le1\}$ are disjoint from $E_3$ for each $k\ge1$ and that $C=\{u+v\mi: u\in \Kc_{1/4}, 0\le v\le 1\}$ is a subset of $E_3$, where $\Kc_{1/4}\subset\bbR$ is the self-similar set determined by $\{x\mapsto\frac{x}{4}, x\mapsto\frac{x+3}{4}\}$. See Figure \ref{leave-2} and verify two basic observations in the following lemmas.

\begin{lemm}\label{E_3-b}
Let $\displaystyle X=\left(\bigcup_{k=9}^{12}f_k(C)\right)\cup\{t+\mi: 1\le t\le2\}$. Then $X$ is a non-locally connected continuum and is a subset of $E_3$.
\end{lemm}

\begin{lemm}\label{E_3-c}
Let $R$ be the rectangle spanned by $B$ in Lemma \ref{E_3-a} and the line segment $\{1+t\mi: 0\le t\le1\}$. Then $f_{12}(R)$ is disjoint from $f_k(Q)$ for $1\le k\le 11$.
\end{lemm}

Now we are ready to prove the following result.

\begin{theo}
The component $P$ of $E_3$ containing $X$ is not locally connected at $2+\frac34\mi$.
\end{theo}
\begin{proof}
Let $U$ denote the open region $\{u+v\mi: \frac58< u<\frac98, -\frac12< v<\frac12\}$. From
Lemma \ref{E_3-c} we can see that $f_{12}(U)\cap E_3$ is a neighborhood of  $2+\frac34\mi$ under the induced topology of $E_3$. The local structure of $E_3$ at  $2+\frac34\mi$ is equivalent to that at $1$. More precisely, if we let $z_k=1-4^{-k}$ for $k\ge1$ then $\{z_k\}$ and $\{f_{12}(z_k)\}$ are sequences in $E_3$ with \[\lim_kf_{12}(z_k)=f_{12}\left(\lim_kz_k\right)=f_{12}(1)=2+\frac34\mi.
 \]
Moreover, none of the points $f_{12}(z_k)$ can be connected to  $2+\frac34\mi$ by any subset of $f_{12}(U)\cap E_3$. This means that the component of  $f_{12}(U)\cap E_3$ that contains $2+\frac34\mi$ is not a neighborhood of $2+\frac34\mi$ in $E_3$. Going to the component $P$ of $E_3$ containing $X$, we can see that  the component of  $f_{12}(U)\cap P$ that contains $2+\frac34\mi$ is not a neighborhood of $2+\frac34\mi$ in $P$. By definition of local connectedness (See Definition \ref{lc}), we have shown that $P$ is not locally connected at $2+\frac34\mi$.
\end{proof}

\noindent{\bf Acknowledgment}. The authors are grateful to Sheng-You Wen at Hubei University for inspiring discussions. They also owe their thanks to the referee(s) for many concrete suggestions, which have greatly improved the paper, especially the proof for Theorem \ref{main-3}.

\bigskip
Jun Luo, Department of Mathematics, Sun Yat-Sen University, Guangzhou, China. (luojun3@mail.sysu.edu.cn)

Hui Rao, Department of Mathematics, Central China Normal University, Guangzhou, China.
(hrao@mail.ccnu.edu.cn)

Ying Xiong, Department of Mathematics,
South China University of Technology, Guangzhou, China. (xiongyng@gmail.com)

\end{document}